\documentclass{amsproc}

\usepackage{amstext,amsthm,amsmath,amsfonts,amssymb,amscd,latexsym,txfonts,stmaryrd}
\usepackage{tikz}
\usepackage{color}
\usepackage{epsfig}
\usepackage{cite}
\usepackage{enumitem}
\usepackage[enableskew]{youngtab}
\usepackage{mathdots}
\usepackage{thmtools}

\usepackage[cmtip,all]{xy}

\usepackage[colorlinks=true, pdfstartview=FitV, linkcolor=blue, citecolor=blue, urlcolor=blue]{hyperref}
\usepackage[nameinlink]{cleveref}

\newcommand{\sgn}{\operatorname{sgn}}

\newcommand{\flo}[1]{\left\lfloor\frac{#1}{2}\right\rfloor}

\newcommand{\Ann}{\operatorname{Ann}}

\newcommand{\C}{{\mathbb C}}
\newcommand{\R}{{\mathbb R}}
\newcommand{\Z}{{\mathbb Z}}

\renewcommand{\P}{{\mathbb P}}

\newtheorem{theorem}{Theorem}[section]
\newtheorem{lemma}[theorem]{Lemma}
\newtheorem{proposition}[theorem]{Proposition}
\newtheorem{corollary}[theorem]{Corollary}
\newtheorem{fact}[theorem]{Fact}

\theoremstyle{definition}
\newtheorem{definition}[theorem]{Definition}
\newtheorem{example}[theorem]{Example}

\theoremstyle{remark}
\newtheorem{remark}[theorem]{Remark}

\numberwithin{equation}{section}

\begin{document}

\title[Hessians and Wronskians]{Some Remarks on Hessians and Wronskians}


\author{Chris McDaniel}
\address{Department of Mathematics, Endicott College, Beverly, Massachussetts 01915}
\email{cmcdanie@endicott.edu}



\date{}
\dedicatory{Dedicated to Tony Iarrobino on the occasion of his $80^{th}$ birthday conference.}

\begin{abstract}
	The purpose of this note is to elaborate on the apparent connection between Wronskians and Hessians.  More generally, to a given subspace of homogeneous bivariate forms over the complex numbers, we associate two determinantal polynomials called the $W$-polynomial and the $\hat{W}$-polynomial.  We give expansion and factorization formulas for these polynomials, and study their behavior under change of coordinates and duality.  As an application, we give another proof of Iarrobino's theorem on the strong Lefschetz property for standard graded Artinian Gorenstein algebras in codimension two.   
\end{abstract}

\maketitle

\section{Introduction}
For fixed integers $r,s$ with $r\leq s+1$, we consider the $s+1$-dimensional complex vector space $\C[X,Y]_s$ consisting of homogeneous bivariate forms of degree $s$, and $r$-dimensional subspaces therein $V\subset\C[X,Y]_s$. Let $N=r(s+1-r)$.  It is convenient to consider $V$ as a point in the Grassmannian $\operatorname{Gr}_r(\C[X,Y]_s)$.  Given a basis of forms $F_1,\ldots,F_r\in V$, we define two determinantal homogeneous bivariate $N$-forms: the $W$-polynomial, denoted by $W(F_1,\ldots,F_r;X,Y)$, and the $\hat{W}$-polynomial, denoted by $\hat{W}(F_1,\ldots,F_r;X,Y)$ (\Cref{def:Wron}).
The $W$-polynomial is a homogenization of the familiar Wronskian polynomial and the $\hat{W}$-polynomial generalizes the Hessian polynomial.  As we shall see, up to a scalar multiple, these two polynomials only depend on the subspace $V$ and not on the particular choice of basis.  In fact, the $W$-polynomial and the $\hat{W}$-polynomial both define morphisms of projective varieties 
$$[W],\left[\hat{W}\right]\colon\operatorname{Gr}_r(\C[X,Y]_s)\rightarrow\P\left(\C[X,Y]_N\right)$$ from the Grassmannian of $r$-dimensional subspaces of homogeneous $s$-forms to the projective space of homogeneous $N$-forms.  It turns out, as we shall see, these two maps are one and the same.

Our main tool is a Pl\"ucker formula for the $W$- and $\hat{W}$-polynomials.
Given a basis $\{F_1,\ldots,F_r\}$ of a subspace $V\in\operatorname{Gr}_r(\C[X,Y]_s)$, denote by $M_V$ its $r\times (s+1)$ coefficient matrix with respect to the usual monomial basis.  We write $W(M_V;X,Y)=W(F_1,\ldots,F_r;X,Y)$, or sometimes just $W(V)$ if we only care about its projective class, and similarly for $\hat{W}$.  For each column indexing $r$-subset $I\subset[s]_0=\{0,1,\ldots,s\}$, let $\Delta_I(M_V)$ be the corresponding minor of $M_V$ (these are the Pl\"ucker coordinates of $V$).  
We show that the $W$ and $\hat{W}$-polynomials satisfy the following ``Pl\"ucker expansion formulas'':
\begin{align}
	\label{eq:PluckW}
	W(M_V;X,Y)=\sum_{I\in\binom{[s]_0}{r}}P_I(X,Y)\cdot\Delta_I(M_V)\\
	\label{eq:PluckhatW}
	\hat{W}(M_V;X,Y)=\sum_{I\in\binom{[s]_0}{r}}Q_I(X,Y)\cdot \Delta_I(M_V)
\end{align}   
where $P_I(X,Y)$ and $Q_I(X,Y)$ are monomial terms depending only on $I$, $r$ and $s$; see \Cref{prop:WhatWPluck}.  Our approach uses weighted path matrices, which allows a combinatorial interpretation of the coefficients of $P_I$ and $Q_I$. 
A Pl\"ucker expansion formula for the Wronskian appears in the work of Purbhoo \cite{Purbhoo}, with applications to Schubert calculus; see also \cite{Karp} and \cite{KP}.  A Pl\"ucker expansion formula for the Hessian appears in the paper \cite{MMS}, where it was used to study Hodge theory on codimension two Artinian Gorenstein algebras.  An earlier incarnation of the Pl\"ucker expansion of the Hessian also appears in Gessel's paper \cite[Theorem 16]{Gessel} in connection with symmetric functions.  

Using \Cref{eq:PluckW,eq:PluckhatW}, we not only show that the $W$-polynomial and the $\hat{W}$-polynomial define the same projective class, but we actually identify the constant of proportionality.  Specifically, we show that there exists a constant $c(r,s)$, depending only on $r$ and $s$ and independent of $V$, satisfying
\begin{equation}
	\label{eq:crs}
	W(M_V;X,Y)=c(r,s)\cdot \hat{W}(M_V;X,Y);
\end{equation}
see \Cref{cor:c}\footnote{After submitting this paper, it was kindly pointed out to us by an anonymous referee that our \Cref{cor:c} was already known to Pasch in 1874!  See \Cref{rem:Pasch} for further details.}.

As a companion to our Pl\"ucker expansion formula, we also have a factorization formula (\Cref{thm:WhatW}) for the projective classes of the $W$- and $\hat{W}$-polynomials based on the following fact (\Cref{prop:WZero}): a linear form $L\in\C[X,Y]_1$ divides $W(M_V;X,Y)$, and hence also $\hat{W}(M_V;X,Y)$, if and only if $L^r\cdot G\in V$ for some $G\in\C[X,Y]_{s-r}$.  The order of division is explained by a certain ``standard basis theorem'' for $V$ with respect to $L$ (\Cref{lem:IY}).
Our factorization formula for the $W$-polynomial, albeit with a slightly different definition, appears in the work of Iarrobino and Yam\'eogo \cite{IY}.	

We define a bilinear pairing on $\C[X,Y]_s$ and associate to each subspace $V\in\operatorname{Gr}_r(\C[X,Y]_s)$ its orthogonal complement subspace $V^\perp\in\operatorname{Gr}_{s+1-r}\left(\C[X,Y]_s\right)$.  The map $V\mapsto V^\perp$ defines an isomorphism of Grassmannians $\operatorname{Gr}_r\left(\C[X,Y]_s\right)\rightarrow\operatorname{Gr}_{s+1-r}\left(\C[X,Y]_s\right)$ (\Cref{lem:DeltaJPerp}).  We show that for each choice of matrix representative $M_V$, there is a choice of matrix representative $M_{V^\perp}$ and a constant $\kappa(r,s)$, depending only on $r$ and $s$ and independent of $V$, for which 
\begin{equation}
	\label{eq:kappa}
	W(M_{V^\perp};X,Y)=\kappa(r,s)\cdot \hat{W}(M_V;X,Y);
\end{equation} 
see \Cref{prop:WPerpHatW}.
The bilinear pairing used to define $V^\perp$ also appears in the work of Karp \cite{Karp}, and implicitly in the work of Iarrobino and Kanev \cite{IK}.

As an application, we prove the following well known result (\Cref{prop:Tony}):
\begin{theorem}
	\label{thm:TonyIntro}
	Every codimension two standard graded Artinian Gorenstein algebra has the strong Lefschetz property.
\end{theorem}
\Cref{thm:TonyIntro} appears in Iarrobino's memoir \cite{I}.  Here we offer two proofs of \Cref{thm:TonyIntro}: one using \Cref{eq:crs}, and a dual version using \Cref{eq:kappa}.  The latter is more or less Iarrobino's original proof.

This paper is organized as follows.  In \Cref{sec:WhatW} we define $W$-polynomials and $\hat{W}$-polynomials.  In \Cref{sec:Plucker} we give the Pl\"ucker expansion formulas and prove our Pl\"ucker expansion formulas in \Cref{prop:WhatWPluck}, as well as \Cref{eq:crs} in \Cref{cor:c}.  In \Cref{sec:Props}, we prove \Cref{prop:WZero} and establish the factorization formula in \Cref{thm:WhatW}.  In \Cref{sec:Dual} we define a bilinear pairing on $\C[X,Y]_s$, we prove \Cref{lem:DeltaJPerp} and we establish \Cref{eq:kappa} in \Cref{prop:WPerpHatW}.   Finally, in \Cref{sec:ExApp}, we compute an example and prove \Cref{thm:TonyIntro} (\Cref{prop:Tony}). 


\section{The $W$-Polynomial and the $\hat{W}$-Polynomial}
\label{sec:WhatW}
Fix positive integers $r,s$ with $r\leq s+1$, and let $F_1,\ldots,F_r\in\C[X,Y]_s$ be $r$ bivariate homogeneous $s$-forms.  
\begin{definition}
	\label{def:Wron}
	Define the \emph{$W$-polynomial} to be the rational function
	\begin{equation}
		\label{eq:Wron}
		W(F_1,\ldots,F_r;X,Y)=\frac{1}{Y^{r(r-1)/2}}\cdot\det\left(\left( \frac{\partial ^{i-1}F_j}{\partial X^{i-1}}\right)_{1\leq i,j\leq r}\right)
	\end{equation}
	Define the \emph{$\hat{W}$-polynomial} to be the determinantal polynomial
	\begin{equation}
		\label{eq:Wrong}
		\hat{W}(F_1,\ldots,F_r;X,Y)=\det\left(\left(\frac{\partial^{r-1}F_j}{\partial X^{i-1}\partial Y^{r-i}}\right)_{1\leq i,j\leq r}\right).
	\end{equation}
\end{definition} 
Clearly the $\hat{W}$-polynomial is a homogeneous polynomial of degree $r(s+1-r)$.  We will show that the $W$-polynomial is also a homogeneous polynomial of the same degree, but first we give some preliminary remarks.

First, note that the specialization of the $W$-polynomial via $Y\mapsto 1$ yields the usual Wronskian:
$$\operatorname{Wr}(F_1(X,1),\ldots,F_r(X,1))=\left(\left(\frac{d^{i-1}F_j(X,1)}{d X^{i-1}}\right)_{1\leq i,j\leq r}\right)=W(F_1,\ldots,F_r;X,1).$$

Also note that if there exists another homogeneous polynomial $F=F(X,Y)$ of degree $d=s+r-1$, and if we define 
$$F_j=\frac{\partial^{r-1}F}{\partial X^{j-1}\partial Y^{r-j}}, \ 1\leq j\leq r,$$
then the $\hat{W}$-polynomial is the $(r-1)^{st}$ Hessian determinant of $F$, i.e.
$$\hat{W}(F_1,\ldots,F_r;X,Y)=\det\left(\left(\frac{\partial ^{2(r-1)}F}{\partial X^{i+j-2}\partial Y^{2r-i-j}}\right)_{1\leq i,j\leq r}\right)=\det\left(\operatorname{Hess}_{r-1}(F)\right).$$

The linearity of the partial differential operators $\partial^{i-1}/\partial X^{i-1}$ and $\partial^{r-1}/\partial X^{i-1}\partial Y^{r-i}$ implies that both the $W$-polynomial and the $\hat{W}$-polynomial depend only on the subspace $V$ spanned by the $F_1,\ldots,F_r$, and not on the polynomials themselves, at least up to a scalar multiple.  Specifically, if $F_1',\ldots,F_r'\in V$ are other forms such that $F_j'=\sum_{i=1}^ra_{ij}F_i$, where the matrix $A=(a_{ij})_{1\leq i,j\leq r}$ is invertible, then we have 
\begin{align*}
	W(F_1',\ldots,F_r';X,Y)= & \det\left(\left(\frac{\partial^{i-1}F_j}{\partial X^{i-1}}\right)_{1\leq i,j\leq r}\cdot\left(a_{i,j}\right)_{1\leq i,j\leq r}\right)\\
	= & W(F_1,\ldots,F_r;X,Y)\cdot\det\left((a_{i,j})_{1\leq i,j\leq r}\right)
\end{align*}
and similarly for the $\hat{W}$-polynomial.
In particular, note that if the forms $F_1,\ldots,F_r$ are linearly dependent, then the $W$-polynomial and the $\hat{W}$-polynomial are both identically zero, i.e. 
$$W(F_1,\ldots,F_r;X,Y)\equiv 0, \ \ \text{and} \ \ \hat{W}(F_1,\ldots,F_r;X,Y)\equiv 0.$$
It turns out that the converse is also true, but this is not so obvious; see \Cref{cor:nonzero}.

Our first task is to verify that the $W$-polynomial is indeed a polynomial and not just a rational function.
\begin{lemma}
	\label{lem:WronPoly}
	For any forms $F_1,\ldots,F_r\in\C[X,Y]_s$, the $W$-polynomial is a polynomial of degree $N=r(s+1-r)$, i.e. $$W(F_1,\ldots,F_r;X,Y)\in\C[X,Y]_N.$$
\end{lemma}
\begin{proof}
	Let $V=\operatorname{span}_\C\{F_1,\ldots,F_r\}\subset\C[X,Y]_s$.  Then one can show that there exists a unique, strictly decreasing sequence of positive integers, depending only on $V$, say $s\geq n_1(V)>\cdots>n_r(V)\geq 0$, and a basis of $V$, say $\{G_1,\ldots,G_r\}$, such that 
	$G_i=Y^{n_i(V)}\cdot H_i$, for some $H_i\in\C[X,Y]_{s-n_i(V)}$ satisfying $H_i(X,0)\neq 0$; see \Cref{lem:IY}.  Given such a basis, note that since multiplication by $Y$ and partial differentiation with respect to $X$ commute, we can write
	\begin{align*}
		W(G_1,\ldots,G_r;X,Y)= & \frac{1}{Y^{\binom{r}{2}}}\cdot \det\left(\left(\frac{\partial^{i-1}G_j}{\partial X^{i-1}}\right)_{1\leq i,j\leq r}\right)\\
		= & \frac{1}{Y^{\binom{r}{2}}}\cdot \det\left(\left(\frac{\partial^{i-1}H_j}{\partial X^{i-1}}\right)_{1\leq i,j\leq r}\cdot\left(\begin{array}{ccc} Y^{n_1(V)} & \cdots & 0\\ 
			\vdots	& \ddots & \vdots\\
			0 & \cdots & Y^{n_r(V)}\\ \end{array}\right)\right)\\
		=\frac{Y^{\sum_{i}n_i}}{Y^{\binom{r}{2}}}\cdot \det\left(\left(\frac{\partial^{i-1}H_j}{\partial X^{i-1}}\right)_{1\leq i,j\leq r}\right).
	\end{align*}
	Since $\sum_{i}n_i\geq \binom{r}{2}=\sum_{i}(i-1)$, it follows that $W(G_1,\ldots,G_r;X,Y)$, and hence $W(F_1,\ldots,F_r;X,Y)$ is a polynomial.
\end{proof}


\begin{remark}
	\label{rem:TonyDef}
	The definition of Wronskian given in \cite[Definition 2.5]{IY} agrees with our definition of $W$-polynomial given above after specializing their variables $dx\mapsto 1$, $dy\mapsto 0$, and scaling by the sign $(-1)^{r(r+1)/2}$.
\end{remark}

\section{Pl\"ucker Expansion Formulas}
\label{sec:Plucker}
\subsection{Weighted Path Matrices}
A directed graph $\Gamma=(\mathcal{V},\mathcal{E})$ is vertex set $\mathcal{V}$ together with a collection $\mathcal{E}$ of ordered pairs of vertices called directed edges; if $e=(a,b)\in \mathcal{E}$ we say that $e$ is directed from $a$ to $b$, where $a$ is the initial vertex and $b$ is the terminal vertex.  A directed path is a sequence $(a_1,a_2,\ldots,a_n)$ where $(a_i,a_{i+1})\in \mathcal{E}$, and a directed cycle is a directed path in which $a_1=a_n$.  We say $\Gamma$ is acyclic if it has no directed cycles.  Given a ring $\mathcal{R}$, an $\mathcal{R}$-edge weighting on $\Gamma$ is a function $\omega\colon \mathcal{E}\rightarrow \mathcal{R}$.  Given an $\mathcal{R}$-weighted directed acyclic graph $(\Gamma,\omega)$, and given vertex subsets of possibly different (finite) cardinalities $\mathcal{A}=\{A_1,\ldots,A_m\},\mathcal{B}=\{B_1,\ldots,B_n\}\subset\mathcal{V}$
define its weighted path matrix to be the $m\times n$ matrix
$$\mathcal{W}(\mathcal{A},\mathcal{B})=\left(\sum_{P\colon A_i\rightarrow B_j}\omega(P)\right)_{\substack{1\leq i\leq m\\ 1\leq j\leq n\\}}$$
where the sum is taken over all directed paths $P$ from $A_i$ to $B_j$, and the weight of the path is the product of its directed edge weights, i.e. $\omega(P)=\prod_{e\in P}\omega(e)$.  Define a path system $\mathcal{P}\colon \mathcal{A}\rightarrow\mathcal{B}$ to be a collection of directed paths $P_i\colon A_i\rightarrow B_{\sigma(i)}$ for injective function $\sigma\colon[m]\rightarrow[n]$, and define its weight to be the product of weights of each of its paths, i.e. $\omega(\mathcal{P})=\prod_{i=1}^n\omega(P_i)$.  We say that the path system is vertex disjoint if no two paths in the system share a common vertex.  In case $m=n$, the path system defines a permutation, and we can define the sign of a path system to be the sign of that permutation, i.e. $\operatorname{sgn}(\mathcal{P})=\operatorname{sgn}(\sigma)$.  The following fundmental result is due to Lindstr\"om \cite{L} and Gessel-Viennot \cite{GV}.
\begin{fact}
	\label{fact:LGV}
	Given an $\mathcal{R}$-weighted directed acyclic graph $(\Gamma,\omega)$, and vertex sets $\mathcal{A}$, and $\mathcal{B}$ as above with $m=n$ (i.e. same cardinality), we have 
	$$\det\left(\mathcal{W}(\mathcal{A},\mathcal{B})\right)=\sum_{\substack{\mathcal{P}\colon\mathcal{A}\rightarrow\mathcal{B}\\
			\text{vertex disjoint}\\}}\operatorname{sgn}(\mathcal{P})\cdot\omega(\mathcal{P})$$
	where the sum (in $\mathcal{R}$) is over all vertex-disjoint path systems from $\mathcal{A}$ to $\mathcal{B}$.	
\end{fact}

\subsection{Matrix Factorizations and Pl\"ucker Expansion Formulas}
Take $\Gamma=(\mathcal{V},\mathcal{E})$ to be the plane lattice whose directed paths are the NE lattice paths, i.e. vertices are lattice points $\mathcal{V}=\Z^2=\{(a,b) \ | \ a,b\in\Z\}$ and directed edges are unit steps up (N) or right (E), i.e. $\mathcal{E}=\{(a,b)\overset{N}{\shortrightarrow}(a,b+1) \ \text{or} \ (a,b)\overset{E}{\shortrightarrow} (a+1,b)\}$.

Define the following finite set of lattice points in the plane, with $r$ and $s$ fixed positive integers, 
\begin{align*}
	\mathcal{A}= & \left\{A_i=(i-1,-i+1) \ | \ 1\leq i\leq r\right\}\\
	\mathcal{B}= & \left\{B_j=(j,0) \ | \ 0\leq j\leq s \right\}\\
	\mathcal{C}= & \left\{C_j=(j,-j+s-r+1) \ | \ 0\leq j\leq s\right\}
\end{align*}
and consder the $r\times (s+1)$-matrices
\begin{align*}
	\mathcal{W}(\mathcal{A},\mathcal{B})= & \left(\binom{j}{j-i+1}\left(\frac{X}{Y}\right)^{j-i+1}Y^{i-1}\right)_{\substack{1\leq i\leq r\\ 0\leq j\leq s\\}}\\
	\hat{\mathcal{W}}(\mathcal{A},\mathcal{C})= & 
	\left(\binom{s-r+1}{j-i+1}X^{j-i+1}Y^{s-r-j+i}\right)_{\substack{1\leq i\leq r\\ 0\leq j\leq s\\}}.
\end{align*}
Taking $\mathcal{R}=\C(X,Y)$ the field of rational functions in $X$ and $Y$ with complex coefficients, the matrix $\mathcal{W}(\mathcal{A},\mathcal{B})$ is the weighted path matrix for the $\mathcal{R}$-weighted directed acyclic graph of NE lattice paths from vertices $\mathcal{A}$ to $\mathcal{B}$ whose East and North edges are weighted by rational functions $\omega(E(a,b))=X/Y$ and $\omega(N(a,b))=Y$, respectively.

Likewise, the matrix $\hat{\mathcal{W}}(\mathcal{A},\mathcal{C})$ is the weighted path matrix for NE lattice paths from vertices $\mathcal{A}$ to $\mathcal{C}$ whose East and North edges are weighted by the variables $\hat{\omega}(E(a,b))=X$ and $\hat{\omega}(N(a,b))=Y$, respectively.

Let $\delta_{ij}$ be the Kronecker delta function, and define the $r\times r$ diagonal matrix
$$D=\left(\delta_{ij}\cdot (i-1)!\cdot Y^{s-3(i-1)}\right)_{1\leq i,j\leq r}$$
and similarly define the $(s+1)\times (s+1)$ diagonal matrix 
$$\hat{D}=\left(\delta_{ij}\cdot \frac{i!\cdot (s-i)!}{(s-r+1)!}\right)_{0\leq i,j\leq s}.$$

Let $V\subset\C[X,Y]_s$ be an $r$-dimensional subspace, and let $\{F_1,\ldots,F_r\}\subset V$ be a basis, with $$F_i=\sum_{j=0}^sm_{ij}X^jY^{s-j}, \ 1\leq i\leq r$$
and define the $r\times (s+1)$ matrix 
$$M_V=\left(m_{ij}\right)_{\substack{1\leq i\leq r\\ 0\leq j\leq s\\}}.$$
Then it is straightforward to check that we have the following factorizations of the underlying rational function matrices for the $W$- and $\hat{W}$-polynomials:
\begin{align}
	\label{eq:WFact}
	\left(\frac{1}{Y^{i-1}}\cdot\frac{\partial^{i-1}F_j}{\partial X^{i-1}}\right)_{1\leq i,j\leq r}= & D\cdot \mathcal{W}(\mathcal{A},\mathcal{B})\cdot M_V^T\\
	\label{eq:hatWFact}
	\left(\frac{\partial^{r-1}F_j}{\partial X^{i-1}\partial Y^{r-i}}\right)_{1\leq i,j\leq r}= & \hat{\mathcal{W}}(\mathcal{A},\mathcal{C})\cdot \hat{D}\cdot M_V^T
\end{align}
where $M^T$ is the transpose of $M$.  For an $r\times (s+1)$ matrix $M$, and for a subset $I\subset\{0,\ldots,s\}=[s]_0$, $|I|=r$, denote by $\Delta_I(M)$ the determinant of the $r\times r$ submatrix of $M$ whose columns are indexed by $I$.  According the the Lindstr\"om-Gessel-Viennot \Cref{fact:LGV}, if $\mathcal{W}(\mathcal{A},\mathcal{B})$ is a weighted path matrix from vertex sets $\mathcal{A}$ to $\mathcal{B}$, then we have  
$$\Delta_I(\mathcal{W}(\mathcal{A},\mathcal{B}))=\sum_{\substack{\mathcal{P}\colon \mathcal{A}\rightarrow\mathcal{B}(I)\\ \text{vertex disjoint}\\}}\sgn(\mathcal{P})\cdot\omega(\mathcal{P})$$
where the sum is over vertex disjoint path systems from vertex set $\mathcal{A}=\{A_1,\ldots,A_r\}$ and vertex subset $\mathcal{B}(I)=\{B_{i_1},\ldots,B_{i_r}\}\subset\mathcal{B}$ indexed by $I$.  Note that for NE lattice paths from $\mathcal{A}$ to $\mathcal{B}$ or $\mathcal{C}$ as above, \emph{vertex disjoint} path systems always define the identity permutation.  Moreover in that case, the weight of a path system satisfies 
\begin{align*}
	\omega(\mathcal{P}\colon \mathcal{A}\rightarrow\mathcal{B}(I))= & 
	\left(\frac{X}{Y}\right)^{|\lambda(I)|}\cdot Y^{r(r-1)/2}\\
	\hat{\omega}(\mathcal{P}\colon \mathcal{A}\rightarrow\mathcal{C}(I))= & 
	X^{|\lambda(I)|}\cdot Y^{r(s-r+1)-|\lambda(I)|}
\end{align*}
where $\lambda(I)$ is the partition associated to $I$, i.e. if $I=\{0\leq i_1<\cdots < i_r\leq s\}$, then $\lambda(I)=(\lambda_1,\ldots,\lambda_r)$ where $\lambda_j=i_{r-j}-(r-j-1)$.  It follows that we have 
\begin{align}
	\label{eq:WAB}
	\Delta_I(\mathcal{W}(\mathcal{A},\mathcal{B}))= & \det\left(\mathcal{W}(\mathcal{A},\mathcal{B}(I))\right)=N(I)\cdot\left(\frac{X}{Y}\right)^{|\lambda(I)|}\cdot Y^{r(r-1)/2}\\
	\label{eq:WAC}
	\Delta_I(\hat{\mathcal{W}}(\mathcal{A},\mathcal{C}))= & \det\left(\hat{\mathcal{W}}(\mathcal{A},\mathcal{C}(I))\right)=\hat{N}(I)\cdot X^{|\lambda(I)|}\cdot Y^{r(s-r+1)-|\lambda(I)|}
\end{align} 
where $N(I)$, respectively $\hat{N}(I)$, is the number of vertex disjoint NE lattice path systems from $\mathcal{A}$ to $\mathcal{B}(I)$, respectively $\mathcal{C}(I)$.  

\begin{remark}
	\label{rem:combinatorial}
	The numbers $N(I)$ and $\hat{N}(I)$ have alternative combinatorial interpretations in terms of the partition $\lambda(I)$:  $N(I)$ counts the number of semi-standard Young tableaux of shape $\lambda(I)$ using the numbers $[r]$ and $\hat{N}(I)$ counts the number of semi-standard Young tableaux of conjugate shape $\lambda'(I)$ using $[s-r+1]$.  See \cite[Section 7.16]{Stanley} for a proof of a similar statement. 
\end{remark}
Given $V\in\operatorname{Gr}_r\left(\C[X,Y]_s\right)$, a basis $\{F_1,\ldots,F_r\}$ with coefficient matrix $M_V$, we write $W(M_V)$ to mean $W(F_1,\ldots,F_r;X,Y)$ and similarly for $\hat{W}$.  
\begin{proposition}[Pl\"ucker Expansion Formulas]
	\label{prop:WhatWPluck}
	With $V$ and $M_V$ defined as above, we have
	\begin{align*}
		W(M_V)= & \sum_{I\in\binom{[s]_0}{r}}\left(\prod_{i=1}^r(i-1)!\cdot N(I)\cdot X^{|\lambda(I)|}\cdot Y^{r(s-r+1)-|\lambda(I)|}\right)\cdot \Delta_I(M_V)\\
		\hat{W}(M_V)= & \sum_{I\in\binom{[s]_0}{r}}\left(\prod_{i\in I}\frac{i!\cdot (s-i)!}{(s-r+1)!}\cdot \hat{N}(I)\cdot X^{|\lambda(I)|}\cdot Y^{r(s-r+1)-|\lambda(I)|}\right)\cdot\Delta_I(M_V)
	\end{align*}
\end{proposition}
\begin{proof}
	The proof is an immediate consequence of \Cref{eq:WFact}, \Cref{eq:hatWFact}, \Cref{eq:WAB}, \Cref{eq:WAC}, and the Cauchy-Binet formula.  
\end{proof}

The next result can be derived combinatorially from \Cref{prop:WhatWPluck} as well as \Cref{eq:WAB,eq:WAC}.
\begin{proposition}
	\label{cor:c}
	Fix positive integers $r\leq s+1$, and set 
	$$c(r,s)=\prod_{i=0}^{r-1}\frac{(s-r+1)!}{(s-i)!}.$$
	Then for any $r$ homogeneous $s$-forms $F_1,\ldots,F_r\in\C[X,Y]_s$ we have the polynomial identity 
	$$W(F_1,\ldots,F_r;X,Y)\equiv c(r,s)\cdot \hat{W}(F_1,\ldots,F_r;X,Y).$$
\end{proposition}
\begin{proof}
	It will suffice to show, by \Cref{prop:WhatWPluck}, that for each $r$-subset $I=\{0\leq i_1<\cdots<i_r\leq s\}$, we have 
	\begin{equation}
		\label{eq:NhatNI}
		\prod_{i=1}^r(i-1)!\cdot N(I)=c(r,s)\cdot \prod_{i\in I}\frac{i!\cdot (s-i)!}{(s-r+1)!}\cdot \hat{N}(I).
	\end{equation}
	According to \Cref{eq:WAB,eq:WAC}, $N(I)$ and $\hat{N}(I)$ are binomial determinants: 
	\begin{align*}
		N(I)= & \det\left(\mathcal{W}(\mathcal{A},\mathcal{B}(I))|_{(X,Y)=(1,1)}\right)=\det\left(\left(\binom{i_k}{i_k-j+1}\right)_{1\leq j,k\leq r}\right)\\
		\hat{N}(I)= & \det\left(\mathcal{W}(\mathcal{A},\mathcal{C}(I))|_{(X,Y)=(1,1)}\right)=\det\left(\left(\binom{s-r+1}{i_k-j+1}\right)_{1\leq j,k\leq r}\right).
	\end{align*}
	It follows that if we regard the products in \Cref{eq:NhatNI} as determinants of diagonal matrices, then \Cref{eq:NhatNI} is equivalent to the equality
	\begin{equation}
		\label{eq:dets}
		\det\left(\left(\binom{s-j+1}{i_k-j+1}\right)_{1\leq j,k\leq r}\right)=\det\left(\left(\binom{s-r+1}{i_k-j+1}\right)_{1\leq j,k\leq r}\right).
	\end{equation}
	Interpreting these matrices as weighted path matrices, note that the LHS of \Cref{eq:dets} counts the number of vertex disjoint path systems from $\mathcal{A}'=\left\{A'_j=(j-1,-r+1) \ | \ 1\leq j\leq r\right\}$ to $\mathcal{B}(I)=\left\{B_{i_k}=(i_k,-i_k+s-r+1) \ | \ 1\leq k\leq r\right\}$, whereas the RHS of \Cref{eq:dets} counts the number of vertex disjoint path systems from $\mathcal{A}=\{(j-1,j-1) \ | \ 1\leq j\leq r\}$ to $\mathcal{B}(I)$.  By truncation, the vertex disjoint path systems from $\mathcal{A}'$ to $\mathcal{B}(I)$ are in bijective correspondence with vertex disjoint path systems from $\mathcal{A}$ to $\mathcal{B}(I)$, and equality \Cref{eq:dets}, and hence \Cref{eq:NhatNI}, holds.
\end{proof}

\begin{remark}
	\label{rem:Pasch}
	\Cref{cor:c} was proved by M. Pasch in his 1874 paper \cite[page 180]{Pasch} using a computation similar to ours, but without the Pl\"ucker expansion formulas and the combinatorial interpretation.  This result also appears in the work of A. Abdesselam and J. Chipalkatti in the broader context of classical invariant theory \cite[Section 2.9]{AC}.  We thank the anonymous referee for pointing out these two references.
\end{remark}

\section{Factorization Formulas}
\label{sec:Props}
\subsection{$\operatorname{GL}_2(\C)$-actions}
There is a natural action of $\operatorname{GL}_2(\C)$ on $\C[X,Y]$ given by linear change of coordinates according to the usual rule: for $F(X,Y)\in\C[X,Y]$ and $\phi=\left(\begin{array}{cc} a & b\\ c & d\\ \end{array}\right)\in\operatorname{GL}_2(\C)$, 
\begin{equation}
	\label{eq:phiaction}
	\phi\left(F\right)(X,Y)=F\left((X,Y)\cdot \left(\begin{array}{cc}a & b\\ c & d\\ \end{array}\right)\right)=F(aX+cY,bX+dY).
\end{equation}

One major difference between $W$-polynomials and $\hat{W}$-polynomials is that the entries in the defining matrix for $\hat{W}$-polynomials all have the same degree.  This allows us to exploit a certain duality, related to Macaulay duality, in studying the effects of linear changes of coordinates on $\hat{W}$-polynomials.

Define the standard graded polynomial ring $R=\C[x,y]$ and define an action of $R$ on the standard graded polynomial ring $S=\C[X,Y]$ by partial differentiation, denoted by
$$x\circ F=\frac{\partial F}{\partial X}, \ \text{and} \ y\circ F=\frac{\partial F}{\partial Y};$$
in general we shall use lower case letters for partial differential operators in $R$, and upper case letters for polynomials in $S$, and the symbol $f\circ F$ shall mean apply the partial differential operator $f\in R$ to the polynomial $F\in S$.  This defines a pairing of graded vector spaces 
\begin{equation}
	\label{eq:pairing}
\langle-,-\rangle\colon R\times S\rightarrow\C, \ \left\langle f,F\right\rangle=(f\circ F)(0,0)
\end{equation}
which restricts to a non-degenerate pairing $\langle-,-\rangle_i\colon R_i\times S_i\rightarrow\C$ for each fixed degree $i\geq 0$.

Given a graded algebra map $\phi\colon S\rightarrow S$ defined by \Cref{eq:phiaction} define its adjoint map to be the algebra map $\phi^*\colon R\rightarrow R$ by 
\begin{equation}
	\label{eq:phistaraction}
	\phi^*(f)(x,y)=f\left((x,y)\cdot \left(\begin{array}{cc}a & b\\ c & d\\ \end{array}\right)^T\right)=f\left(ax+by,cx+dy\right)
\end{equation}
where $A^T$ is the transpose of $A$.
The following formula gives the relationship between $\phi$ and $\phi^*$ with respect to the bilinear pairing defined by \Cref{eq:pairing}.
\begin{lemma}
	\label{lem:action}
	Given $\phi\colon S\rightarrow S$ and $\phi^*\colon R\rightarrow R$ as in \Cref{eq:phiaction,eq:phistaraction}, and for any homogeneous elements $f\in R$ and $F\in S$, we have 
	$$\phi\left(\phi^*(f)\circ F\right)=f\circ \phi(F).$$
\end{lemma}
\begin{proof}
	Use induction on $\deg(f)$.  For the base case, assume $f=ex+gy$ is a linear form.  Then for any $F\in\C[X,Y]$, we have	\begin{align*}
		\phi\left(\phi^*(f)\circ F\right)= & \phi\left((ea+gc)\frac{\partial F}{\partial X}+(eb+gd)\frac{\partial F}{\partial Y}\right)\\
		= & (ea+gc)\frac{\partial F}{\partial X}(aX+cY,bX+dY)+(eb+gd)\frac{\partial F}{\partial Y}(aX+cY,bX+dY), 
	\end{align*}
	whereas 
	\begin{align*}
		f\circ \phi(F)= & (ex+gy)\circ F(aX+cY,bX+dY)\\
		= & e\left(a\cdot \frac{\partial F}{\partial X}(aX+cY,bX+dY)+b\cdot \frac{\partial F}{\partial Y}(aX+cY,bX+dY)\right)\\
		& +g\left(c\cdot \frac{\partial F}{\partial X}(aX+cY,bX+dY)+d\cdot \frac{\partial F}{\partial Y}(aX+cY,bX+dY)\right)
	\end{align*}
	(the second equaltiy follows from the chain rule) and they are the same.
	
	For the inductive step, assume the formula holds for $f'\in\C[x,y]$ with $\deg(f')<\deg(f)$.  We can factor $f$ and write $f=\ell\cdot f'$ for some linear form $\ell$ and some homogeneous polynomial $f'\in\C[x,y]$ satisfying $\deg(f')=\deg(f)-1$.  Then we have 
	\begin{align*}
		\phi\left(\phi^*(f)\circ F\right)= & \phi\left(\left(\phi^*(\ell)\cdot \phi^*(f')\right)\circ F\right)\\
		= & \phi\left(\phi^*(\ell)\circ \left(\phi^*(f')\circ F\right)\right)\\
		= & \ell\circ \phi\left(\phi^*(f')\circ F\right)\\
		= & \ell\circ \left(f'\circ \phi(F)\right)=f\circ \phi(F)
	\end{align*}
	where the second to last equality follows from the base case and the last equality follows from our inductive hypothesis.  This completes the inductive step, and hence the proof.
\end{proof}
A consequence of \Cref{lem:action} is the following formula for the $\operatorname{GL}_2(\C)$-action on the $W$- and $\hat{W}$-polynomials.
\begin{lemma}
	\label{cor:hatWaction}
	Let $\phi\colon \C[X,Y]\rightarrow\C[X,Y]$ be any invertible linear change of coordinates, i.e. $\phi\in\operatorname{GL}_2(\C)$, and let $A=(\phi^*)^{-1}\colon\C[x,y]_{r-1}\rightarrow\C[x,y]_{r-1}$ be the linear transformation that is the restriction to the degree $r-1$ graded component of the inverse of the adjoint map.  Then for any $r$ homogeneous $s$-forms $F_1,\ldots,F_r\in\C[X,Y]_s$, we have
	$$\phi\left(\hat{W}(F_1,\ldots,F_r)\right)=\det(A)\cdot \hat{W}\left(\phi(F_1),\ldots,\phi(F_r)\right).$$
	The same holds for the $W$-polynomial. 
\end{lemma}
\begin{proof}
	Of course the last statement follows from the statement for the $\hat{W}$-polynomial and \Cref{cor:c}.  Hence it suffices to show the formula for the $\hat{W}$-polynomial.
	
	First, observe that, given a basis $\{F_1,\ldots,F_r\}\subset V$, we have 
	$$\hat{W}(F_1,\ldots,F_r)=\det\left(\left(x^{i-1}y^{r-i}\circ F_j\right)_{1\leq i,j\leq r}\right).$$
	Then according to \Cref{lem:action}, we have 
	\begin{align*}	\phi\left(\hat{W}(F_1,\ldots,F_r)\right)= &
		\det\left(\phi\left(\phi^*\left((\phi^*)^{-1}(x^{i-1}y^{r-i})\right)\circ F_j\right)\right)\\
		= & \det\left((\phi^*)^{-1}(x^{i-1}y^{r-i})\circ\phi(F_j)\right)\\
		=  & \det(A)\cdot \hat{W}\left(\phi(F_1),\ldots,\phi(F_r)\right)
	\end{align*}
	as claimed, and the proof is complete.	
\end{proof}

\begin{remark}
	\label{rem:Determinant}
	\begin{enumerate}
		\item Note that if $A_{r}$ is a matrix representative for $\phi^*$ restricted to the degree $r-1$ graded component of $\C[x,y]$, then $\det(A_r)=\det(A_2)^{\binom{r}{2}}$.  Indeed, if $\alpha,\beta$ are the eigenvalues of $A_2$, then the eigenvalues of $A_r$ are $\{\alpha^{i-1}\beta^{r-i} \ | \ 1\leq i\leq r\}$ and hence the determinant satisfies
	$$\det(A_r)=\prod_{i=1}^r\alpha^{i-1}\beta^{r-i}=\left(\alpha\beta\right)^{\binom{r}{2}}.$$
	In particular, an alternative, cleaner formulation of \Cref{lem:action} is 
	$$\det(\phi^*)^{\binom{r}{2}}\cdot \phi\left(\hat{W}(F_1,\ldots,F_r)\right)=\hat{W}(\phi(F_1),\ldots,\phi(F_r)).$$
	We thank the anonymous referee for pointing this out to us.  
	\item Note that given $\phi, \phi^*$ as in \Cref{eq:phiaction,eq:phistaraction}, and for any fixed degree $i$, it follows from \Cref{lem:action} that the non-degenerate pairing defined by \Cref{eq:pairing} satisfies 
	$$\langle f,\phi(F)\rangle_i=
	\langle \phi^*(f),F\rangle_i, \ \forall f\in R_i, F\in S_i.$$
	In particular, $\phi^*$ is indeed the adjoint map of $\phi$ with respect to the bilinear pairing $\langle-,-\rangle_i$, for every fixed degree $i\geq 0$.
\end{enumerate}
\end{remark}



\subsection{Factorization formulas}
The following ``standard basis theorem'' for bivariate polynomials generalizes the method in the proof of \Cref{lem:WronPoly}; it is stated as part of \cite[Definition 2.3]{IY}.  The proof is straightforward and left to the reader.  

\begin{lemma}
	\label{lem:IY}
	Let $V\subset\C[X,Y]_s$ be any $r$-dimensional subspace, and let $L=aX+bY\in\C[X,Y]_1$ be any linear form.  Then there exists a strictly decreasing\footnote{In \cite{IY}, the sequence is increasing, but this is only a matter of convention.} sequence of non-negative integers 
	$$s\geq n_1(V,L)>\cdots>n_r(V,L)\geq 0$$ 
	and nonzero homogeneous polynomials $H_i(X,Y)\in\C[X,Y]_{s-n_i(V,L)}$ relatively prime to $L$ such that the set of forms
	$$\left\{F_i=L^{n_i(V,L)}\cdot H_i \ | \ 1\leq i\leq r\right\}$$
	is a basis of $V$. 	
\end{lemma}

In general, strictly decreasing sequences of non-negative integers $s\geq n_1>\cdots>n_r\geq 0$ of length $r$ and size at most $s$, are in one-to-one correspondence with weakly decreasing sequences of non-negative integers $s-r+1\geq \lambda_1\geq \cdots\geq \lambda_r\geq 0$, or partitions, of length at most $r$ and part size at most $s-r+1$, where 
\begin{equation}
	\label{eq:lambda}
	\lambda_i=n_{i}-(r-i), \ 1\leq i\leq r;
\end{equation}
in other words the (Young diagram of the) partition $\lambda=(\lambda_1,\ldots,\lambda_r)$ fits inside an $r\times (s-r+1)$ rectangle.  
Let $\lambda(V,L)$ denote the partition associated to the strictly decreasing sequence $s\geq n_1(V,L)>\cdots>n_r(V,L)\geq 0$.  As usual, the size of the partition is the sum of its parts, which in our case is $$|\lambda(V,L)|=\sum_{i=1}^r\left(\lambda_i(V,L)=n_i(V,L)-(r-i)\right)=\sum_{i=1}^rn_i(V,L)-\frac{r(r-1)}{2}.$$


We shall also need the following classical result of B\^ocher \cite[Theorem II]{Bocher}, and we refer the reader to that paper for its proof.
\begin{fact}
	\label{lem:Bocher1}
	Let $u_1(X),\ldots,u_r(X)$ be any linearly independent set of analytic functions of one variable, e.g. polynomials.  Then the Wronskian determinant is not identically zero, i.e.
	$$\operatorname{Wr}(u_1,\ldots,u_r):=\det\left(\left(\frac{d^{i-1}u_j}{dX^{i-1}}\right)_{1\leq i,j\leq r}\right)\not\equiv 0.$$	
\end{fact}

\begin{corollary}
	\label{cor:nonzero}
	For any $r$-linearly independent homogeneous $s$-forms $F_1,\ldots,F_r\in\C[X,Y]_s$, $W$-polynomial $\hat{W}$-polynomial are each not identically zero, i.e.
	$$W(F_1,\ldots,F_r;X,Y)\not\equiv 0\not\equiv\hat{W}(F_1,\ldots,F_r;X,Y).$$ 
\end{corollary}
\begin{proof}
	If $F_1,\ldots,F_r$ are linearly independent, then so are the specializations $F_1(X,1),\ldots,F_r(X,1)$.  It follows that if $W(F_1,\ldots,F_r;X,Y)\equiv 0$ is identically zero, then so is its specialization $$W(F_1,\ldots,F_r;X,1)\equiv\operatorname{Wr}(F_1(X,1),\ldots,F_r(X,1))\equiv 0,$$ contradicting B\^ocher's \Cref{lem:Bocher1}.  Therefore $W(F_1,\ldots,F_r;X,Y)\not\equiv 0$, and hence also $\hat{W}(F_1,\ldots,F_r;X,Y)\not\equiv 0$ by \Cref{cor:c}.
\end{proof}

The next result characterizes the zeros of the $W$-polynomial; compare with \cite[Lemma 2.7]{IY}.
\begin{proposition}
	\label{prop:WZero}
	Let $V\subset\C[X,Y]_s$ be an $r$-dimensional subspace, and let $L=aX+bY\in\C[X,Y]_1$ be any linear form.  Then 
	\begin{equation}
		\label{eq:WLU}
		W(V)=L^{|\lambda(V,L)|}\cdot U(V,L)
	\end{equation}
	for some homogeneous polynomial $U(V,L)$, which is relatively prime to $L$.  
	In particular, $L$ divides $W(V)$ if and only if $L^r\cdot G\in V$ for some homogeneous polynomial $G\in\C[X,Y]_{s-r}$.
	
	The same holds for $\hat{W}(V)$.
\end{proposition} 
\begin{proof}
	Of course the last statement follows from \Cref{cor:c}, so we focus on the first two statements.  
	First, we show how \Cref{eq:WLU} implies the second statement.  Assume that $L^r\cdot G\in V$.  Then according to \Cref{lem:IY}, we must have $n_1(V,L)\geq r$, and hence $\lambda_1(V,L)=n_1(V,L)-(r-1)\geq 1$, and hence $|\lambda(V,L)|>0$.  Therefore, according to \Cref{eq:WLU}, $L$ must divide $W(V)$.  Conversely, assume that $L$ divides $W(V)$.  Then, working backwards, $|\lambda(V,L)|>1$, according to \Cref{eq:WLU}, and hence $\lambda_1(V,L)=n_1(V,L)-(r-1)\geq 1$.  Therefore, according to \Cref{lem:IY}, we must have $F_1=L^{n_1(V,L)}\cdot H_1=L^r\cdot G\in V$ for some homogeneous polynomial $G$.

	It remains to show \Cref{eq:WLU}.  It suffices to prove it in the case where $L=Y$, by \Cref{cor:hatWaction}.  In this case, let $\{F_1,\ldots,F_r\}\subset V$ be a basis as in \Cref{lem:IY}, where $F_i=Y^{n_i(V,Y)}\cdot H_i$ where $s\geq n_1>\cdots>n_r\geq 0$ and $H_i$ is relatively prime to $Y$.  As we observed in the proof of \Cref{lem:WronPoly}, we have 
	\begin{align*}
		Y^{r(r-1)/2}\cdot W(F_1,\ldots,F_r;X,Y)= & \det\left(\left(\frac{\partial^{i-1}}{\partial X^{i-1}}\left(Y^{n_j}\cdot H_j\right)\right)_{1\leq i,j\leq r}\right)\\
		= & \det\left(\left(\frac{\partial^{i-1}H_j}{\partial X^{i-1}}\right)_{1\leq i,j\leq r}\right)\cdot \det\left(\left(\begin{array}{ccc} Y^{n_1(V,Y)} & \cdots & 0\\ \vdots & \ddots & \vdots\\ 0 & \cdots & Y^{n_r(V,Y)}\\ \end{array}\right)\right)\\
		= & U(V,Y)\cdot Y^{\sum_{i=1}^rn_i(V,Y)} 
	\end{align*}
	Therefore we have 
	$$W(V)=W(F_1,\ldots,F_r;X,Y)=Y^{|\lambda(V,Y)|}\cdot U(V,Y).$$
	It remains to see that the polynomial $U(V,Y)=U(V,Y;X,Y)$ does not vanish at $Y=0$.  Since evaluation at $Y=0$ commutes with the partial differentiation operator $\partial/\partial X$, and since $H_j(X,0)=X^{s-n_j(V,Y)}$, it follows that 
	$$U(V,Y;X,0)=\det\left(\left(\frac{\partial^{i-1}X^{s-n_j(V,Y)}}{\partial X^{i-1}}\right)_{1\leq i,j\leq r}\right)=\operatorname{Wr}(X^{s-n_1(V,Y)},\ldots,X^{s-n_r(V,Y)}).$$
	Since all the $n_i(V,Y)$'s are distinct, it follows that the functions $\{X^{s-n_i(V,Y)} \ | \ 1\leq i\leq r\}$ are linearly independent, hence by B\^ocher's \Cref{lem:Bocher1}, it follows that $U(V,Y;X,0)\not\equiv 0$, and the result follows.
\end{proof}
We are now in a position to establish factorization formulas for the $W$- and $\hat{W}$-polynomials.
\begin{proposition}
	\label{thm:WhatW}
	For every $r$-dimensional subspace $V\subset\C[X,Y]_s$, we have the following equivalences of projective classes of polynomials: 
	$$\left[W(V)\right]=\left[\prod_{[L]\in\P\left(\C[X,Y]_1\right)}L^{|\lambda(V,L)|}\right]=\left[\hat{W}(V)\right]$$
	where the product in the middle is taken over all projective classes of linear forms in $\P(\C[X,Y]_1)$.
\end{proposition}	
\begin{proof}
	The first equality follows immediately from \Cref{prop:WZero}, and the second from \Cref{cor:c}.
\end{proof}

\begin{remark}
	\label{rem:SchubEisHarTony}
	\Cref{prop:WhatWPluck} together with \Cref{cor:nonzero} imply that the $W$- and $\hat{W}$-polynomials define morphisms of projective varieties
	$$[W],\left[\hat{W}\right]\colon\operatorname{Gr}_r\left(\C[X,Y]_s\right)\rightarrow\P\left(\C[X,Y]_N\right),$$
	which are exactly the same, according to \Cref{cor:c}.	
	Moreover, according to Iarrobino and Yam\'eogo \cite[Proposition 2.15]{IY}, this morphism is a finite covering map of degree 
	$$d_{r,s+1}=N!\cdot \frac{1!\cdot 2!\cdots r!}{(s+1-r)!\cdot(s+2-r)!\cdots s!},$$ 
	which can be verified using Schubert calculus.  This result, and its connection to Schubert calculus, seems to have been  first published in 1983 by Eisenbud and Harris \cite[Theorem 2.3]{EH}, although they state there that Iarrobino had independently obtained the same result in an unpublished work ten years earlier.
\end{remark}

\section{Duality}
\label{sec:Dual}
The non-degenerate pairing between $\C[x,y]_s$ and $\C[X,Y]_s$ given by partial differentiation $\langle f,F\rangle_s=f\circ F$ gives a canonical isomorphism $\C[x,y]_s\cong \left(\C[X,Y]_s\right)^*$.  Choosing the linear isomorphism $\psi\colon\C[X,Y]_s\rightarrow\C[x,y]_s$ defined by $\psi(X^iY^{s-i})=(-1)^iy^ix^{s-i}$, i.e. change of coordinates sending $X\to -y$ and $Y\to x$, defines another non-degenerate pairing on $\C[X,Y]_s$ defined by 
\begin{equation}
	\label{eq:Pairing}
	\left\langle F, G\right\rangle_s^{\psi} =\langle \psi(F),G\rangle_s=\psi(F)\circ G.
\end{equation}
Explicitly, if $F=\sum_{j=0}^sa_jX^jY^{s-j}$ and $G=\sum_{j=0}^sb_jX^jY^{s-j}$, then 
$$\langle F,G\rangle_s^{\psi} = s!\cdot \sum_{j=0}^s(-1)^j\frac{a_{j}b_{s-j}}{\binom{s}{j}}.$$

Let $V\subset\C[X,Y]_s$ be any $r$-dimensional subspace, and define its orthogonal complement by 
$$V^\perp=\left\{F\in \C[X,Y]_s \ | \ \langle F,G\rangle_s^{\psi} \equiv 0, \ \forall G\in V\right\}.$$
For any $r$-subset $I\in\binom{[s]_0}{r}$, define the $s+1-r$-subsets
\begin{align*}
	I^c= & \{j\in [s]_0 \ | \ j\notin I\}\\
	I^\perp= & \{k\in [s]_0 \ | \ k=s-j, \ j\in I^c\}
\end{align*}
The following lemma appears in \cite[Lemma 1.11]{Karp0} where it is attributed to Hochster, who in turn has attributed it to Hilbert.  A proof has been included here for the sake of completeness.
\begin{lemma}
	\label{lem:DeltaJPerp}
	Let $M_V$ be an $r\times (s+1)$ matrix representative for an $r$-dimensional subspace $V\subset\C[X,Y]_s$.  Then there exists an $(s+1-r)\times (s+1)$ matrix representative $M_{V^\perp}$ for the subspace $V^\perp$ such that for any $r$-subset $I\in\binom{[s]_0}{r}$ we have
	\begin{equation}
		\label{eq:DeltaMVPerp}
		\Delta_I(M_V)=\prod_{j\in I^c}\binom{s}{j}^{-1}\cdot \Delta_{I^\perp}(M_{V^\perp}).
	\end{equation}	
\end{lemma}
\begin{proof}
	First note that given an $r\times (s+1)$ matrix $M_V$, a column vector $(a_0,\ldots,a_s)^T\in\R^{s+1}$ is in the kernel of $M_V$ if and only if the row vector
	$$(a_0,\ldots,a_s)\cdot P_s$$
	is in the row space of $M_{V^\perp}$, where $P_s$ is the $(s+1)\times (s+1)$ matrix with $(-1)^j\binom{s}{j}$ in the $j^{th}$ column and the $(s-j)^{th}$ row and zeros elsewhere, i.e. $$P_s=\left(\begin{array}{ccccc} 0 & \cdots & 0 & \cdots  & (-1)^s\binom{s}{s}\\
		\vdots & & \vdots & \iddots & \vdots\\
		0 & \cdots & (-1)^j\binom{s}{j} & \cdots & 0\\
		\vdots & \iddots & \vdots & & \vdots\\
		(-1)^0\binom{s}{0} & \cdots & 0 & \cdots & 0\\ \end{array}\right).$$
	If $M_V$ has rank $r$, choose and fix an $(s+1-r)\times (s+1)$ matrix $W_V$ for which the $(s+1)\times (s+1)$ matrix 
	$M=\left(\begin{array}{c} M_V\\
		\hline
		W_V\\ \end{array}\right)$ is invertible of determinant $1$.  Then its inverse satisfies $M^{-1}=\left(\begin{array}{c|c} W_V^* & \tilde{M}_{V}^T\\ \end{array}\right)$, where $\tilde{M}_V$ is an $(s+1-r)\times (s+1)$ matrix satisfying $M_V\cdot\tilde{M}_V^T=0$, and hence $\tilde{M}_{V}\cdot P_s=M_{V^\perp}$.  Then in particular, for any $k$, $1\leq k\leq  s+1$, and for any $k$-subsets $I,J\in\binom{[s+1]}{k}$ (note for the moment we are shifting our indexing to start at $1$ instead of zero), Jacobi's determinantal identity states
	\begin{equation}
		\label{eq:Jacobi}
		\Delta_{IJ}\left(M\right)=(-1)^{\sum I+\sum J}\cdot \det(M)\cdot \Delta_{J^cI^c}(M^{-1}).
	\end{equation}
	Choosing $k=r$ and $I=[r]$, we get 
	\begin{align*}
		\Delta_J(M_V)= & (-1)^{\binom{r+1}{2}+\sum J}\cdot \Delta_{J_c}(\tilde{M}_V=M_{V^\perp}\cdot P_s^{-1})\\
		= & (-1)^{\binom{r+1}{2}+\sum J}\cdot \Delta_{J^\perp}(M_{V^\perp})\cdot \left(\prod_{j\in J^c}(-1)^{s-j+1}\binom{s}{j-1}^{-1}\right)\cdot (-1)^{\flo{s+1-r}}\\
		= & (-1)^{\binom{r+1}{2}+\sum J+\sum(s+1-J^c)+\flo{s+1-r}}\cdot \prod_{j\in J^c}\binom{s}{j-1}^{-1}\cdot\Delta_{J^\perp}(M_{V^\perp})\\
		= & (-1)^{\binom{r+1}{2}+\binom{s+2}{2}+(s+1)(s+1-r)+\flo{s+1-r}}\cdot \prod_{j\in J^c}\binom{s}{j-1}^{-1}\cdot \Delta_{J^\perp}(M_{V^\perp}).
	\end{align*}
	Shifting our indexing back down to zero, i.e. setting $I=J-1$, and noting that the sign is even, i.e. 
	$$\binom{r+1}{2}+\binom{s+2}{2}+(s+1)(s+1-r)+\flo{s+1-r}\equiv 0 \ \text{mod} \ 2$$
	yields the desired result.
\end{proof}

For any $r$-subset $I\in\binom{[s]_0}{r}$, we have 
\begin{equation}
	\label{eq:ILambda}
	\lambda(I^c)=\tilde{\lambda}(I)'
\end{equation}
where $I_c\in\binom{[s]_0}{s+1-r}$ is the complement of $I$ and $\tilde{\lambda}=(s+1-r-\lambda_r,\ldots,s+1-r-\lambda_1)$ is the complementary partition to $\lambda$ inside the $r\times (s+1-r)$ rectangle, and $\tilde{\lambda}'=(\tilde{\lambda}_1',\ldots,\tilde{\lambda}'_{s+1-r})$ is its conjugate partition, i.e.
$$\tilde{\lambda}_k'=\#\left\{j \ | \ \tilde{\lambda}_j\geq k\right\}.$$
To see this, suppose that $I=\left\{0\leq i_1<\cdots<i_r\leq s\right\}$ and $I^c=\left\{0\leq b_1<\cdots<b_{s+1-r}\leq s\right\}$ so that $I\sqcup I^c=[s]_0$.  Note that for each $j$, we have 
$$b_j-(j-1)=\#\{k \ | \ i_k<b_j\}=\#\{k \ |  \ i_k-(k-1)\leq j+1\}.$$  
It follows that for each $j$, $1\leq j\leq s+1-r$ we have 
\begin{align*}
	\lambda(I^c)_j= & b_{s+1-r-j}-(s-r-j)\\
	= & \#\{k \ | \ i_k<b_{s+1-r-j}\}\\
	= & \#\{k \ | \ i_k-(k-1)\leq s+1-r-j\}\\
	= & \#\{k \ | \ s+1-r-\lambda(I)_{r-k}\geq j\}\\
	= & \#\{k \ | \ \tilde{\lambda}(I)_k\geq j\}\\
	= & \tilde{\lambda}(I)'_{j}
\end{align*}
as claimed.
Also, it is straightforward to see that $\lambda(I^\perp)=\tilde{\lambda}(I^c)$, and hence it follows that we have 
$$\lambda(I^\perp)=\tilde{\lambda}(I^c)=\lambda(I)'.$$
In particular, it follows, from their description in terms of partitions in \Cref{rem:combinatorial}, that we have $N(I^\perp)=\hat{N}(I)$.  Putting this together with \Cref{lem:DeltaJPerp} and \Cref{prop:WhatWPluck}, we have proved the following result.
\begin{proposition}
	\label{prop:WPerpHatW}
	With $V\subset\C[X,Y]_s$, $M_V$ and $M_{V^\perp}$ defined as above, we have 
	$$W(M_{V^\perp};X,Y)=\kappa(r,s)\cdot \hat{W}(M_V;X,Y)$$
	where $$\kappa(r,s)=\frac{\prod_{i=0}^s\binom{s}{i}\prod_{i=1}^{s-r+1}(i-1)!}{\left(s(s-1)\cdots(s-r+2)\right)^r}.$$
\end{proposition}

\begin{remark}
	\label{rem:Morphism}
	\Cref{lem:DeltaJPerp} implies that the correspondence $V\mapsto V^\perp$ defines an isomorphism of Grassmannian varieties $\operatorname{Gr}_r\left(\C[X,Y]_s\right)\rightarrow\operatorname{Gr}_{s+1-r}\left(\C[X,Y]_s\right)$, and  \Cref{prop:WPerpHatW}, together with \Cref{cor:c}, shows that under this isomorphism, the projective class of the $\hat{W}$-polynomial (and hence also the $W$-polynomial) is preserved.
\end{remark}

\section{Example and Application}
\label{sec:ExApp}	
\begin{example}
	\label{ex:GR24}
	Take $r=2$ and $s=4$, and let $V\subset\C[X,Y]_4$ be the $2$-dimensional subspace spanned by 
	\begin{align*}
		G_1= & Y^4+2XY^3-X^3Y\\
		G_2= & X^2Y^2+X^3Y+X^4
	\end{align*}
	so that 
	$$M_V=\left(\begin{array}{rrrrr} 1 & 2 & 0 & -1 & 0\\ 0 & 0 & 1 & 1 & 1\\ \end{array}\right).$$
	We can extend $M_V$ to an invertible matrix with determinant one by, say,
	$$M=\left(\begin{array}{rrrrr} 1 & 2 & 0 & -1 & 0\\ 0 & 0 & 1 & 1 & 1\\
		0 & 1 & 0 & 0 & 0\\ 0 & 0 & 0 & 1 & 0\\ 0 & 0 & 0 & 0 & -1\\ \end{array}\right)$$
	which has inverse
	$$M^{-1}=\left(\begin{array}{rrrrr} 1 & 0 & -2 & 1 & 0\\ 0 & 0 & 1 & 0 & 0\\ 0 & 1 & 0 & -1 & 1\\ 0 & 0 & 0 & 1 & 0\\ 0 & 0 & 0 & 0 & -1\\ \end{array}\right)$$
	so that $\tilde{M}_V=\left(\begin{array}{rrrrr}
		-2 & 1 & 0 & 0 & 0\\ 1 & 0 & -1 & 1 & 0\\ 0 & 0 & 1 & 0 & -1\\ \end{array}\right)$.  It follows that 
	$$M_{V^\perp}=\left(\begin{array}{rrrrr} 0 & 0 & 0 & -4 & -2\\ 0 & -4 & -6 & 0 & 1\\ -1 & 0 & 6 & 0 & 0\\ \end{array}\right)$$
	corresponding to the $s-r+1=3$-dimensional subspace $V^\perp\subset\C[X,Y]_4$ spanned by 
	\begin{align*}
		F_1= & -4X^3Y-2X^4\\
		F_2= & -4XY^3-6X^2Y^2+X^4\\
		F_3= & -Y^4+6X^2Y^2.
	\end{align*}
	Then we compute 
	\begin{align*}
		\hat{W}(M_V;X,Y)= & \det\left(\begin{array}{rr}4Y^3+6XY^2-X^3 & 2X^2Y+X^3\\
			2Y^3-3X^2Y & 2XY^2+3X^2Y+4X^3\\ \end{array}\right)\\
		= & - 4X^6 + 28X^4Y^2 + 32X^3Y^3 + 20X^2Y^4 + 8XY^5
	\end{align*}
	and 
	\begin{align*}
		W(M_{V^\perp};X,Y)= & \frac{1}{Y^3}\cdot \det\left(\begin{array}{rrr} 
			-4X^3Y-2X^4 & -4XY^3-6X^2Y^2+X^4 & -Y^4+6X^2Y^2\\
			-12X^2Y-8X^3 & -4Y^3-12XY^2+4X^3 & 12XY^2\\
			-24XY-24X^2 & -12Y^2+12X^2 & 12Y^2\\ \end{array}\right)\\
		= & - 48X^6 + 336X^4Y^2 + 384X^3Y^3 + 240X^2Y^4 + 96XY^5
	\end{align*}
	and 
	\begin{align*}
		\kappa(r,s)=\kappa(2,4)= & \frac{\binom{4}{0}\binom{4}{1}\binom{4}{2}\binom{4}{3}\binom{4}{4}(0!)(1!)(2!)}{4^2}=12,
	\end{align*}
	and indeed we have $W(M_{V^\perp})=\kappa(2,4)\cdot \hat{W}(M_V)$.  We can also compute 
	\begin{align*}
		W(M_V;X,Y)= & \frac{1}{Y}\cdot \det\left(\begin{array}{rr}
			Y^4+2XY^3-X^3Y & X^2Y^2+X^3Y+X^4\\ 
			2Y^3-3X^2Y & 2XY^2+3X^2Y+4X^3\\ \end{array}\right)\\
		= & - X^6 + 7X^4Y^2 + 8X^3Y^3 + 5X^2Y^4 + 2XY^5
	\end{align*}
	and 
	\begin{align*}
		c(r,s)=c(2,4)=\frac{(3!)\cdot (3!)}{(4!)\cdot (3!)}=\frac{1}{4}
	\end{align*}
	and we have $W(M_V)=c(2,4)\cdot \hat{W}(M_V)$, as predicted by \Cref{cor:c}.
\end{example}

For a fixed positive integer $d$, and fixed homogeneous $d$-form $F\in\C[X,Y]_d$, let $\Ann(F)\subset\C[x,y]=R$ denote the annihilator ideal, and $A=A_F=R/\Ann(F)$ the codimension two graded Artinian Gorenstein algebra of socle degree $d$.  By Macaulay duality theory, every graded Artinian Gorenstein algebra $A$ of socle degree $d$ defined over $\C$ has the form $A=A_F$ for some $F\in\C[X,Y]$.  Recall that a graded Artinian Gorenstein algebra of socle degree $d$ satisfies the strong Lefschetz property if there exists a linear form $\ell\in A_1$ for which the multiplication maps $\times\ell^{d-2i}\colon A_i\rightarrow A_{d-i}$ are isomorphisms for each $0\leq i\leq \flo{d}$.  The following result was proved by Iarrobino \cite[Theorem 2.9]{I}.
\begin{proposition}
	\label{prop:Tony}
	Every codimension two graded Artinian Gorenstein algebra $A$ satisfies the strong lefschetz property.
\end{proposition}
\begin{proof}
	According to a theorem of Macaulay (see \cite[Theorem 1.44]{IK}), the algebra $A_F$ is a complete intersection with Hilbert function of the form 
	$$H(A_F)=(1,2,\ldots,s(F)-1,s(F)^k,s(F)-1,\ldots,2,1).$$
	Thus it suffices to show that for each $r$ satisfying $1\leq r\leq s(F)\leq \flo{d}+1$, the multiplication map $\times\ell^{d-2(r-1)}\colon\left(A_F\right)_{r-1}\rightarrow\left(A_F\right)_{d-r+1}$ is an isomorphism.  Fix such an $r$ and let $s=d-r+1$, and define the $r$-dimensional subspace $$V=V^F_r=\operatorname{span}_{\C}\left\{\left.F_j=\frac{\partial^{r-1}F}{\partial X^{j-1}\partial Y^{r-j}}\right| 1\leq j\leq r\right\}\subset\C[X,Y]_s$$
	to be the subspace of partial derivatives of $F$ of order $(r-1)$.  Then one can show directly that 
	$$\hat{W}(F_1,\ldots,F_r)=(-1)^{\flo{r}}\cdot \det\left(\operatorname{Hess}_{r-1}(F)\right)$$
	where $\operatorname{Hess}_{r-1}(F)$ is the $(r-1)^{st}$ Hessian matrix for $F$, i.e.
	$$\operatorname{Hess}_{r-1}(F)=\left(\frac{\partial^{2(r-1)}F}{\partial X^{p+q-2}\partial Y^{2r-p-q}}\right)_{1\leq p,q\leq r}.$$  
	
	According to a theorem of J. Watanabe \cite{Watanabe}, the $(r-1)^{st}$ Hessian matrix evaluated at $(X,Y)=(a,b)$ is the matrix for the $(r-1)^{st}$ Lefschetz multiplication map $$\times(ax+by)^{d-2(r-1)}\colon \left(A_F\right)_{r-1}\rightarrow \left(A_F\right)_{d-r+1}.$$
	According to \Cref{cor:nonzero}, $\hat{W}(V^F_r;X,Y)=\hat{W}(F_1,\ldots,F_r;X,Y)$ is not identically zero, and hence the nonzero set $D^F_r=\left\{(a,b)\in\C^2 \ | \ \hat{W}(V^F_r;a,b)\neq 0\right\}$ is a non-empty Zariski open set.  Since this holds for each $r$ with $1\leq r\leq s(F)$, there must be some linear form in the intersection of finitely many non-empty Zariski open sets $\ell=ax+by\in \bigcap_{r=1}^{s(F)}D^F_r\subseteq \left(A_F\right)_1$ that satisfies the strong Lefschetz property on $A_F$.
	
	Alternatively, one can also give a dual proof using the $W$-polynomial, along the same lines as Iarrobino's original argument.  Indeed, define the ring automorphism $$\theta=\psi^{-1}\colon\C[x,y]\rightarrow\C[X,Y], \ \theta(x)=Y, \ \theta(y)=-X.$$ 
	Then one can show that
	$$V^\perp= \theta\left(\Ann(F)_{d+1-r}\right).$$
	Since $\dim_\C\left(\left(V^F_r\right)^\perp\right)=s+1-r=d-2(r-1)$, it follows from \Cref{prop:WZero} that, for a linear form $\ell\in R_1$, the multiplication map $\times\ell^{d-2(r-1)}\colon\left(A_F\right)_{r-1}\rightarrow\left(A_F\right)_{d-r+1}$ has nonzero kernel if and only if $\theta(\ell)$ is a factor of the $W$-polynomial $W\left(\left(V^F_r\right)^\perp;X,Y\right)$.  Since $W\left(\left(V^F_r\right)^\perp;X,Y\right)$ is not identically zero, it can only have finitely many linear factors.  It follows that there exists a linear form $\ell$ such that $\theta(\ell)$ is not a factor of $W\left(\left(V^F_r\right)^\perp;X,Y\right)$ for any $1\leq r\leq s(F)$, and this linear form must be strong Lefschetz for $A_F$.  	
\end{proof}

\section*{Acknowledgments}
The author gratefully acknowledges the anonymous referees for their careful reading and helpful remarks.  The author is also grateful to Tony Iarrobino, Pedro Macias Marques, Alexandra Seceleanu and Junzo Watanabe for many useful and inspiring discussions on Hessians and Wronskians and many other related topics.  

\bibliographystyle{amsplain}

\begin{thebibliography}{99}
	\bibitem{AC}
	Abdesselam, A., Chipalkatti, J.,
	\emph{On the Wronskian combinants of binary forms},
	Journal of Pure and Applied Algebra, 210, 43--61, (2007).
	
	\bibitem{Bocher}
	Bôcher, M.,
	\emph{Certain Cases in Which the Vanishing of the Wronskian Is a Sufficient Condition for Linear Dependence}. 
	Transactions of the American Mathematical Society 2, no. 2 (1901).
	
	\bibitem{EH}
	Eisenbud, D., Harris, J.,
	\emph{Divisors on general curves and cuspidal rational curves},
	Invent. math., 74, 371--418 (1983).
	
	\bibitem{Fulton}
	Fulton, W.,
	\emph{Intersection Theory, second edition}.
	Springer, 1998.
	
	\bibitem{Gessel}
	Gessel, I.,
	\emph{Symmetric functions and P-recursiveness},
	Journal of Combinatorial Theory, Series A,
	Volume 53, Issue 2,
	1990,
	Pages 257-285.
	
	\bibitem{GV}
	Gessel, I., Viennot, X.,
	\emph{Binomial determinants, paths and hook length formulae},
	\textit{I. Gessel} and \textit{G. Viennot}, Adv. Math. 58, 300--321 (1985).
	
	\bibitem{I}
	Iarrobino, A.,
	\emph{Associated Graded Algebra of a Gorenstein Artin Algebra}.
	Memoirs of the AMS, Vol. 107, No. 514, Providence, RI (1994)
	
	\bibitem{IK}
	Iarrobino, A., Kanev, V.,
	\emph{Power sums, Gorenstein algebras, and determinantal loci. With an appendix `The Gotzmann theorems and the Hilbert scheme' by Anthony Iarrobino and Steven L. Kleiman}. Lecture Notes in Mathematics, 1721, Springer (1999).
	
	\bibitem{IY}
	Iarrobino, A., Yam\'eogo, J.,
	\emph{The Family G T of Graded Artinian Quotients of $k[x, y]$ of Given Hilbert Function}. 
	Communications in Algebra, 31(8), 3863–3916 (2003).
	
	\bibitem{Karp0}
	Karp, S.,
	\emph{Sign variation, the Grassmannian, and total positivity},
	Journal of Combinatorial Theory, Series A,
	Volume 145,
	Pages 308-339, (2017).
	
	\bibitem{Karp}
	Karp, S.,
	\emph{Wronskians, total positivity, and real Schubert calculus}. 
	Sel. Math. New Ser. 30, 1 (2024). 
	
	\bibitem{KP}
	Karp, S., Purbhoo, K.,
	\emph{Universal Pl\"ucker Coordinates for the Wronski Map and Positivity in Real Schubert Calculus} (preprint).
	arXiv:2309.04645 (2023).
	
	\bibitem{L}
	Lindstr\"om, B.,
	\emph{On the Vector Representations of Induced Matroids},
	Bulletin of The London Mathematical Society, 5(1), 85–90 (1973). 
	
	\bibitem{MMS}
	Macias Marques, P., McDaniel, C., Seceleanu, A., 
	\emph{Higher Lorentzian Polynomials, Higher Hessians, and the Hodge–Riemann Relations for Graded Oriented Artinian Gorenstein Algebras in Codimension Two}, 
	International Mathematics Research Notices, Volume 2025, Issue 13, July 2025.
	
	\bibitem{Errata}
	Macias Marques, P., McDaniel, C., Seceleanu, A., 
	\emph{Errata to ``Higher Lorentzian Polynomials, Higher Hessians, and the Hodge–Riemann Relations for Graded Oriented Artinian Gorenstein Algebras in Codimension Two''} (preprint)
	
	\bibitem{Pasch}
	Pasch, M., 
	\emph{Note uber die Determinanten, welche aus Functionen und deren Differentialen gebildet werden}, J. Reine Angew. Math., 80, 177--182, (1875).
	
	\bibitem{Purbhoo}
	Purbhoo, K.,
	\emph{Jeu de taquin and a monodromy problem for Wronskians of polynomials},
	Adv. Math. 224, No. 3, 827--862 (2010).
	
	\bibitem{Stanley}
	Stanley, R.,
	\emph{Enumerative Combinatorics Volume 2}
	Cambridge Studies in Advanced Mathematics 62, Cambridge University Press, 2005.
	
	\bibitem{Watanabe}
	Watanabe, J.,
	\emph{A remark on the Hessian of homogeneous polynomials}, The curves seminar at Queen’s, vol. XIII, Queen’s Papers in Pure and Appl. Math., vol. 119, 2000,
		Queen’s Univ. Kingston, ON, pp. 171–178.
	
	
\end{thebibliography}

\end{document}